\documentclass{article}
\usepackage{relsize}
\usepackage{color}
\usepackage{mathtools}
\usepackage[dvipsnames,table]{xcolor}
\usepackage[noadjust]{cite}

\usepackage{hyperref, enumitem}

\hypersetup{
  colorlinks   = true, 
  urlcolor     = blue, 
  linkcolor    = Purple, 
  citecolor   = red 
}
\usepackage{amsmath,amsthm,amssymb}
\usepackage{xurl}
\hypersetup{breaklinks=true}
\usepackage[margin=2.5cm]{geometry}

\usepackage{graphicx}
\usepackage{algorithm}
\usepackage{algorithmic}
\usepackage[T1]{fontenc}
\usepackage{authblk}
\usepackage[english]{babel}
\usepackage{stackrel}

\usepackage[capitalize,noabbrev]{cleveref}

\usepackage[normalem]{ulem}
\usepackage{cancel}

\usepackage{tikz}
\usetikzlibrary{calc}
\usetikzlibrary{decorations.pathreplacing}
\usetikzlibrary{shapes.geometric,quotes}
\usetikzlibrary{shapes.misc}

\tikzset{cross/.style={cross out, draw=black, minimum size=2*(#1-\pgflinewidth), inner sep=0pt, outer sep=0pt}, 
cross/.default={1pt}}
\theoremstyle{definition}
\newtheorem{theorem}{Theorem}[section]

\newtheorem{proposition}[theorem]{Proposition}
\newtheorem{corollary}[theorem]{Corollary}
\newtheorem{lemma}[theorem]{Lemma}
\newtheorem{definition}[theorem]{Definition}
\newtheorem{example}[theorem]{Example}
\newtheorem{remark}[theorem]{Remark}

\newcommand{\F}{\mathbb F}

\newcommand{\N}{\mathbb N}
\newcommand{\Fq}{\F_q}

\newcommand{\mC}{\mathcal C}
\newcommand{\C}{\mathcal C}
\newcommand{\mD}{\mathcal D}

\newcommand{\mT}{\mathcal T}

\newcommand{\mS}{\mathcal S}
\newcommand{\mU}{\mathcal U}

\DeclareMathOperator{\rk}{rk}

\DeclareMathOperator{\PG}{PG}

\DeclareMathOperator{\GL}{GL}
\DeclareMathOperator{\supp}{supp}
\DeclareMathOperator{\wt}{wt}
\DeclareMathOperator{\dd}{d}
\DeclareMathOperator{\HH}{H}
\DeclareMathOperator{\rs}{rowsp}
\DeclareMathOperator{\cs}{colsp}

\newcommand\qqbin[2]{\left[\begin{matrix} #1 \\ #2 \end{matrix} \right]}

\title{A geometric approach to generalized covering radii of linear codes}

\author[1]{Gianira N. Alfarano}
\author[2]{Giuseppe Marino}
\author[2]{Alessandro Neri}
\author[2]{Rocco Trombetti}

\affil[1]{Universit\'e de Rennes, IRMAR.}
\affil[2]{Department of Mathematics and Applications ``R. Caccioppoli'', University of Naples Federico II.}

\affil[ ]{{\small \texttt{gianira-nicoletta.alfarano@univ-rennes.fr},
\texttt{giuseppe.marino@unina.it},}}\affil[ ]{{\small
 \texttt{alessandro.neri@unina.it},
\texttt{rocco.trombetti@unina.it}}}
\date{}

\begin{document}

\maketitle

\begin{abstract}
Covering problems in coding theory are closely related to finite geometry through the interpretation of the columns of parity-check matrices as point sets in finite vector spaces. Motivated by the recent notion of generalized covering radii of linear codes introduced by Elimelech, Firer and Schwartz, we develop a geometric framework for these parameters. We introduce $(\rho,t)$-saturating sets and show that they are precisely the finite-geometric counterparts of linear codes whose $t$-th generalized
covering radius is at most~$\rho$.
We study the structure of these sets and show that the extremal case $\rho=t$ coincides with the notion of $t$-strong blocking sets. Thus, $(\rho,t)$-saturating sets interpolate between classical saturating sets and strong blocking sets. We provide several equivalent formulations, including affine and dual Grassmannian criteria, derive lower bounds on their size, and give constructions from strong blocking sets, graphs and projective configurations.
\end{abstract}

\medskip



\section{Introduction}
Covering problems are a classical topic in coding theory; see e.g. \cite{cohen1997covering}. Given a linear code $\mC\subseteq \F_q^n$, its covering radius measures the smallest integer $\rho$ such that every vector of $\F_q^n$ lies at Hamming distance at most $\rho$ from a codeword of $\mC$. Equivalently, if $H$ is a parity-check matrix of $\mC$, the covering radius is the smallest $\rho$ such that every syndrome can be expressed as a linear combination of at most $\rho$ columns of $H$. This second interpretation provides a natural bridge with finite geometry: the columns of a parity-check matrix may be viewed as vectors of a finite vector space, and covering properties of the code become spanning properties of these vectors; see e.g. \cite{davydov2003saturating, davydov2002constructions, davydov2005locally}.

The \emph{generalized covering radii} of a linear code were recently introduced by Elimelech, Firer and Schwartz as a natural generalization of the classical covering radius in \cite{elimelech2021generalized}. Instead of asking whether each syndrome can be generated by a small number of columns of a parity-check matrix, one asks whether every set of $t$ syndromes can be generated using the same set of at most $\rho$ columns. Thus, the $t$-th generalized covering radius $R_t(\mC)$ is the smallest integer $\rho$ such that every $t$-element subset of the syndrome space is contained in the span of at most $\rho$ columns of a parity-check matrix of $\mC$. The case $t=1$ recovers the ordinary covering radius, which has been intensively investigated in numerous papers and books; see e.g. \cite{cohen1997covering, cohen1985covering, guruswami2005complexity, bartoli2014covering} and references therein.

The introduction of these parameters was motivated in part by applications to linear data-querying protocols, including private information retrieval. In such settings, one stores suitable linear combinations of the database entries, and the covering radius controls the number of stored combinations that must be accessed in order to answer a query. If several queries are grouped together, then the relevant question is no longer how to cover one syndrome at a time, but how to cover a set of syndromes using a common set of stored combinations. This is precisely the role of generalized covering radii. Beyond this motivation, generalized covering radii admit several equivalent descriptions already proposed in \cite{elimelech2021generalized}. Initial work on generalized covering radii focused both on the general theory and on the computation of these parameters for special families of codes. In particular, the generalized covering radii of Reed--Muller codes
were investigated in~\cite{elimelech2022generalized}, where lower and upper bounds were proved in
several asymptotic regimes and exact values were determined in certain extremal cases. More recently, generalized covering radii have also been
studied for BCH codes. In~\cite{elimelech2024second}, the second generalized covering radius of binary primitive double-error-correcting BCH codes was determined, while
in~\cite{ozbudak2026third} the third generalized covering radius of the same family was studied.

The aim of this paper is to develop a finite-geometric approach to generalized covering radii. To this end, we introduce the notion of a \emph{$(\rho,t)$-saturating set}; see Definition \ref{def:rho_t_saturating}. In this language, we say that a linear code has $t$-th generalized covering radius at most $\rho$ if and only if the set of columns of a parity-check matrix forms a $(\rho,t)$-saturating set. 
This notion generalizes the classical theory of saturating sets, which we recover for $t=1$. Such objects were introduced by Ughi in \cite{ughi1987saturated} and have since been widely studied in finite geometry and coding theory, especially through their connection with covering codes; see e.g. \cite{giulietti2013geometry, denaux2022constructing, davydov2011linear}.

A second finite-geometric object naturally appearing in this framework is that of \emph{strong blocking sets}. Introduced in \cite{davydov2011linear}, a set $\mS$ is called a $t$-strong blocking set if, for every $t$-dimensional subspace $\Lambda$, the intersection $\mS\cap \Lambda$ spans $\Lambda$. Strong blocking sets have recently attracted attention because of their relation with \emph{minimal codes} and related extremal problems in finite geometry; see \cite{alfarano2022geometric, bishnoi2024blocking, chen2025t}. In our setting they arise as the extremal case of our generalized saturation property: a set is $(t,t)$-saturating if and only if it is a $t$-strong blocking set; see Proposition~\ref{prop:strong_(t,t)saturating}. Thus $(\rho,t)$-saturating sets also generalize strong blocking sets.

These generalizations suggest a unified approach to several covering and blocking phenomena. On the one hand, strong blocking sets can be used to construct generalized saturating sets, and hence linear codes with controlled generalized covering radii. On the other hand, lower bounds for $(\rho,t)$-saturating sets may be obtained by counting the number of $t$-dimensional subspaces that can be covered by $\rho$-dimensional subspaces generated by a prescribed set of vectors. When $\rho=t$, this returns to the theory of strong blocking sets and its connection with affine blocking sets. When $\rho>t$, however, the relevant problem is no longer an ordinary point-blocking condition, but a higher-dimensional covering problem for $t$-subspaces by generated $\rho$-subspaces.

In this paper we formalize the connection between generalized covering radii and finite geometry. We first recall the coding-theoretic definitions of generalized covering radii and introduce $(\rho,t)$-saturating sets. We then prove the equivalence between $(\rho,t)$-covering linear codes and $(\rho,t)$-saturating sets arising from parity-check matrices. We show that the special case $\rho=t$ is precisely the theory of $t$-strong blocking sets, and we provide several geometric reformulations of the general notion, including affine and dual Grassmannian criteria. We also derive lower bounds on the size of $(\rho,t)$-saturating sets and present constructions obtained from strong blocking sets, graphs and projective-geometric configurations.

The paper is organized as follows. In Section \ref{sec:prel} we recall the necessary background on generalized covering radii, saturating sets, and strong blocking sets. In Section \ref{sec:geom} we introduce $(\rho,t)$-saturating sets and prove their equivalence with $(\rho,t)$-covering linear codes. We also establish several equivalent geometric formulations, including affine and dual Grassmannian criteria, and identify the case $\rho=t$ with $t$-strong blocking sets. In Section \ref{sec:low_bound} we prove a general lower bound on the size of $(\rho,t)$-saturating sets. Finally, in Section \ref{sec:constructions} we present constructions arising from strong blocking sets and from projective configurations associated with graphs.

\medskip

\section*{Acknowledgments}
This research has been partially supported by the Italian National Group for Algebraic and Geometric Structures and their Applications (GNSAGA - INdAM), by Università Italo Francese (UIF/UFI) via PHC Galileo 2026 - G26-260/54322VM and by the University of Naples Federico II, which funded the ``Naples-Rennes Agreement on Scientific Cooperation''.
G. N. Alfarano is supported by the Agence Nationale de la Recherche through grant number ANR-24-CPJ1-0075-01. A. Neri is supported by the INdAM - GNSAGA Project CUP E53C24001950001  ``Noncommutative polynomials in coding theory''.

\medskip

\section{Preliminaries}\label{sec:prel}
In this section we provide the necessary background material for the rest of the paper. Throughout the paper we will use the following notation.

\paragraph{\textbf{Notation.}} Let $q$ be a prime power, and let $\Fq$ denote the finite field with $q$ elements. Let  $r,k,t,n \in \N$ be fixed positive integers. Denote by $[n]:=\{1,\ldots,n\}$. 
For a matrix $A \in \Fq^{r\times n}$ and a subset $I\subseteq [n]$, we denote by $A_I\in \Fq^{r\times |I|}$ the submatrix of $A$ obtained by restricting only to the columns indexed by $I$. We write $\rs_{\Fq}(A)$ and $\cs_{\Fq}(A)$ to indicate the $\Fq$-span of the rows and, respectively, of the columns of $A$. For a given vector subspace $V\subseteq \F_q^k$, we denote by $\PG(V)$ its corresponding projective space. In particular, $\PG(k-1,q)$ denotes the projective space with underlying vector space $\F_q^k$.

\subsection{Linear codes and generalized covering radii}\label{subsec:gen_cov_radii}

In this subsection we recall the notion of linear codes, their generalized covering radii and other related concepts. For more details, we refer the interested reader to \cite{elimelech2021generalized}.

Let $v \in \Fq^n$. The \textbf{support} of $v$ is the set $$\supp(v):=\{i \,:\, v_i \neq 0\}\subseteq [n].$$
Moreover, for a subset $\mC\subseteq \Fq^n$, the support of $\mC$ is
$$\supp(\mC)=\bigcup_{v \in \mC}\supp(v).$$
The \textbf{Hamming} \textbf{weight} of a vector $v\in\F_q^n$ is $\mathrm{wt}_{\HH}(v)=|\supp(v)|$ and
the \textbf{Hamming distance} on $\Fq^n$ is defined as 
$$\dd_{\HH}(u,v)=\wt_{\HH}(u-v).$$

An $[n,n-k]_q$ \textbf{(linear) code} is a  $(n-k)$-dimensional $\F_q$-linear subspace $\mC \subseteq \F_q^n$. The elements of $\mC$ are called \textbf{codewords} and the vectors in $\F_q^k$ are called \textbf{syndromes}. $\mC$ is said to be \textbf{nondegenerate} if there exist no $i\in[n]$ such that $c_i=0$ for every $c\in\mC$.
The \textbf{minimum distance} of
$\mC$ is the integer 
$$d(\mC)=\min\{\mathrm{wt}_{\HH}(c) \mid c \in \mC, \, c \neq 0\}.$$
If $d=d(\mC)$ is known, we say that $\mC$ is an $[n,n-k,d]_q$ code. A matrix $G\in\F_q^{(n-k)\times n}$ is called a \textbf{generator matrix} of $\mC$ if $\rs_{\F_q}(G)=\mC$. A matrix $H\in\F_q^{k\times n}$ is a \textbf{parity-check matrix} of $\mC$ if $\mC=\ker(H)$, i.e. 
$$\C:=\{v\in\F_q^n \; : \; vH^\top = 0\}.$$
The $[n,k]_q$ code generated by a parity-check matrix $H$ is called \textbf{dual code} of $\mC$ and it is denoted by $\mC^\perp$.

Finally, codes $\mathcal{C}$ and $\mathcal{C}^\prime$ are called (\textbf{monomially})  \textbf{equivalent} if there exists an $\F_q$-linear isometry $f: \F_q^n \to \F_q^n$ with $f(\mC)=\mC'$; see~\cite[page 24]{huffman2010fundamentals}.

The covering radius of a code $\mC$ is the maximum distance of $\mC$ to any vector in the ambient space, or the minimum value $\rho$ such that the union of the spheres of radius $\rho$ around each codeword cover the ambient space.
We now recall the definition of generalized covering radii given in \cite{elimelech2021generalized} in the form that will be most useful for the geometric interpretation developed later. If $H$
is a parity-check matrix of a linear code $\mC$, then the
columns of $H$ are vectors of $\F_q^k$ that determine which syndrome can be obtained using a prescribed set of coordinates.

\begin{definition}\label{def:gcr1}
    Let $H\in \Fq^{k\times n}$ be a parity-check matrix of an $[n,n-k]_q$ code $\C$ and let $t \in \N$. The \textbf{$t$-th generalized covering radius} of $\C$ is defined as
    $$R_t(\C):=\max\limits_{\substack{\mS\subseteq \Fq^{k} \\ |\mS|=t}}\min\limits_{\substack{I\subseteq [n] \\ \mS\subseteq \cs_{\Fq}(H_I)}}|I|.$$
\end{definition}

Note that the above definition does not depend on the choice of $H$. When $t=1$, Definition \ref{def:gcr1} retrieves the notion of the classical covering radius of the code. Moreover,  it can be easily seen that $$0\le R_1(\C)\leq R_2(\C)\le \cdots \le R_{k}(\C)=k.$$

In \cite{elimelech2021generalized}, the authors provided several equivalent definitions of generalized covering radii for linear codes. For the sake of completeness, we report them here, but we are not going to use anything different from Definition \ref{def:gcr1}.

For a vector $v \in \Fq^n$ and a positive integer $r$, the \textbf{ball of radius $r$ and center $v$} is the set
$$B_{r}(v):=\{u \in \Fq^n\,:\, \dd_{\HH}(u,v)\le r\}.$$
Moreover, for a subset $I\subseteq [n]$ and a vector $v$, the \textbf{cube with support $I$ and center $v$} is the set
$$Q_I(v):=\{u \in \Fq^n \,:\, \supp(u-v)\subseteq I\}.$$
Observe that 
$$B_r(v)=\bigcup_{I\in \binom{[n]}{r}}Q_I(v).$$

\begin{definition}\label{def:gcr2}
    Let $\C$ be an $[n,n-k]_q$ code. The \textbf{$t$-th generalized covering radius} of $\C$ is 
    $$\min \{ r\in\mathbb{N} \,:\, \forall v_1,\ldots,v_t\in\Fq^{n}, \exists c_1,\ldots, c_t \in \C, I \in \binom{[n]}{r} \mbox{ s.t. } v_i \in Q_I(c_i) \mbox{ for each } i \in [t]\}. $$
\end{definition}

Another interesting definition of generalized covering radii can be given as follows.

\noindent For a given $[n,n-k]_q$ code $\C$ and a positive integer $t$, consider the code 
$$\C_t:=\C\otimes \F_{q^t}=\langle \C\rangle_{\F_{q^t}},$$
that is, the code generated by $\C$ over the extension field $\F_{q^t}$.

\begin{definition}\label{def:gcr4}
    The \textbf{$t$-th generalized covering radius} of an $[n,n-k]_q$ code $\C$ is defined as
    $$R_t(\C):=R_1(\C_t).$$
\end{definition}

\begin{proposition}[{\cite[Lemma 7]{elimelech2021generalized}}]
       Definitions \ref{def:gcr1}, \ref{def:gcr2} and \ref{def:gcr4} coincide.
\end{proposition}

Whenever linear codes are investigated with the goal of understanding properties related to the covering radius, such codes are called \emph{covering codes}. More precisely, we introduce the following definition.

\begin{definition}
Let $\rho$ be a positive integer. We say that a linear code $\mC$ is a \textbf{$(\rho,t)$-covering code} if its $t$-th generalized covering radius is at most $\rho$. A $(\rho,1)$-covering code is simply a $\rho$-\textbf{covering code} and in this case, $\rho$ is the classical \textbf{covering radius}; see \cite{ughi1987saturated} for more references.
\end{definition}

\subsection{Classical saturating sets and strong blocking sets}\label{subsec:sat_sets_sbs}
Covering radii have a classical geometric counterpart: the columns of a parity-check matrix may be regarded as vectors of a finite vector space, and the covering problem becomes a spanning problem. This leads to the notion of \emph{saturating sets}, which has been introduced in \cite{ughi1987saturated} and subsequently investigated in many other works, where several constructions based on combinatorial and geometric tools were also presented; see \cite{davydov2002constructions, davydov2003saturating, davydov2011linear, denaux2022constructing} to mention but a few.

\begin{definition}\label{def:saturating_classic}
Let $\mS=\{h_1,\ldots,h_n\} \subseteq \F_q^{k}$ be a subset of vectors. 
\begin{enumerate}
    \item[(1)] A vector $v\in\F_q^{k}$ is \textbf{$\rho$-saturated} by $\mS$ if there exists a subspace of dimension at most $\rho$ containing $v$ that is spanned by vectors in $\mS$.
    \item[(2)] $\mS$ is \textbf{$\rho$-saturating} if all the vectors of $\F_q^{k}$ are $\rho$-saturated by $\mS$.
\end{enumerate}
\end{definition}

A stronger geometric condition, which has proved useful in constructing saturating sets, is obtained by requiring not merely that vectors are covered,
but that every subspace of a prescribed dimension is generated by the vectors of
the set lying inside it. This leads to \emph{strong blocking sets}, introduced in \cite{davydov2011linear}.

\begin{definition}
    Let $r\leq k$ be a positive integer. A subset $\mS\subseteq \F_q^{k}$ is an \textbf{$r$-strong blocking set} if for every~$r$-dimensional subspace $\Lambda$ of $\F_q^{k}$ we have $\langle \mS\cap \Lambda\rangle = \Lambda$.
\end{definition}

Precisely, the following result holds.

\begin{theorem}[{\cite[Theorem 3.2]{davydov2011linear}}] \label{thm:sbs-saturating}
Let $\rho\leq k$ be a positive integer. Then any $\rho$-strong blocking set in $\F_{q}^{k}\subseteq \F_{q^\rho}^{k}$ is a $\rho$-saturating set in $\F_{q^\rho}^{k}$.
    
\end{theorem}

\begin{remark}
    We point out that if $\mS\subseteq\F_q^{k}$ is an $r$-strong blocking set, for $r\leq k$, then for every $r\leq s\leq k$,~$\mS$ is also an $s$-strong blocking set.
\end{remark}

Recently, $r$-strong blocking sets in $\F_q^k$ have been connected with $(k-r)$-minimal codes; see \cite{chen2025t} for more details. 

\begin{definition}
Let $\mC$ be an $[n,k]_q$ code, and let $0\leq s\leq k-1$. We say that $\mC$ is \textbf{$s$-minimal} if every subcode $\mD\subseteq \mC$ of dimension $s$
is minimal, that is, if for every subcode $\mD'\subseteq \mC$ such that $\supp(\mD')\subseteq \supp(\mD)$ and
$\dim(\mD')=\dim(\mD)$, one has $\mD'=\mD$.
\end{definition}

In particular, in \cite[Theorem 4.5]{chen2025t} it is shown that there is a one-to-one correspondence between monomial equivalence classes of nondegenerate $[n,k]_q$ $(k-r)$-minimal codes and equivalence classes of $r$-strong blocking sets.
This generalizes the result of \cite[Theorem 3.4]{alfarano2022geometric} for $t=k-1$.

\section{Geometric approach to generalized covering radii}\label{sec:geom}
In this section we introduce a geometric approach to the study of generalized covering radii, which generalizes the concept of saturating sets.

\subsection{\texorpdfstring{$(\rho,t)$}{(rho,t)}-saturating sets and covering codes}\label{subsec:relations_covering_sat}

Classical $\rho$-saturating sets give the case in which one wants to
cover one vector at a time. For generalized covering radii, one has to cover several vectors simultaneously using the same small set of columns. This
leads to the following definition, which generalizes Definition \ref{def:saturating_classic}.

\begin{definition}\label{def:rho_t_saturating}
Let $\mS=\{h_1,\ldots,h_n\}\subseteq \F_q^{k}$.
\begin{enumerate}
    \item[(1)] A set $X\subseteq \F_q^{k}$ is \textbf{$\rho$-saturated} by $\mS$ if there exists a subspace of dimension at most $\rho$ containing $X$ that is spanned by vectors in $\mS$.
    \item[(2)]  Let $t,\rho$ be positive integers, with $\rho\geq t$. $\mS$ is \textbf{$(\rho,t)$-saturating} if every set $X\subseteq\F_q^{k}$ with $|X|=t$ is $\rho$-saturated by $\mS$.
\end{enumerate}
\end{definition}

Note that any $(\rho,t)$-saturating set $\mS$ clearly has to span the whole ambient space.
We now illustrate two elementary properties that will be used implicitly in some constructions. We say that two $(\rho,t)$-saturating sets $\mS_1$ and $\mS_2$ are \textbf{(linearly) equivalent} if there exists an $\F_q$-linear isomorphism $f$ of $\F_q^{k}$ such that $f(\mS_1) = \mS_2$

\begin{remark}\label{rem:saturating} Let $\mS$ be a $(\rho,t)$-saturating set. Then:
\begin{enumerate}
    \item[(1)] $\mS$ is $(\rho,s)$-saturating for every $s\le t$.
    \item[(2)] $\mS$ is a $(s\rho,st)$-saturating for every $s\ge 1$.
\end{enumerate}
\end{remark}

We now prove the basic correspondence between the coding-theoretic and the geometric points of view. Under this correspondence, the columns of a parity-check matrix form a $(\rho,t)$-saturating set exactly when the associated code has $t$-th generalized covering radius at most $\rho$.

\begin{theorem}
\label{thm:saturating-covering-correspondence}
Let $\mS=\{h_1,\ldots,h_n\}\subseteq \Fq^{k}$, and assume that the matrix
$H=(h_1\ \cdots\ h_n)\in \Fq^{k\times n}$ has $\rk(H)=k$.
Let $\mC=\ker(H)$. Then $\mS$ is a $(\rho,t)$-saturating set if and only if
$\mC$ is an $[n,n-k]_q$ $(\rho,t)$-covering code.
Moreover, this correspondence induces a one-to-one correspondence between
equivalence classes of $(\rho,t)$-saturating sets in $\Fq^{k}$ and monomial
equivalence classes of nondegenerate $[n,n-k]_q$ $(\rho,t)$-covering codes.
\end{theorem}

\begin{proof}
Since $\rk(H)=k$, the code $\mC=\ker(H)$ has dimension $n-k$, and $H$ is a parity-check matrix of $\mC$.
Let $T=\{s_1,\ldots,s_t\}\subseteq \Fq^{k}$. By definition, $T$ is $\rho$-saturated by $\mS$ if and only if there exists a subset $I\subseteq [n]$, with $|I|\leq \rho$, such that
$$
T\subseteq \langle h_i : i\in I\rangle_{\Fq}.
$$
Since the columns of $H_I$ are precisely the vectors $h_i$ with $i\in I$, this is equivalent to saying that $T\subseteq \cs_{\Fq}(H_I)$.
Therefore, $\mS$ is $(\rho,t)$-saturating if and only if for every $t$-element subset $T\subseteq \Fq^{k}$ there exists $I\subseteq [n]$, with $|I|\leq \rho$, such that $T\subseteq \cs_{\Fq}(H_I)$. By Definition~\ref{def:gcr1}, this is equivalent to require $R_t(\mC)\leq \rho$. Hence $\mS$ is $(\rho,t)$-saturating if and only if $\mC$ is a $(\rho,t)$-covering code.

It remains to check that the correspondence is compatible with equivalence. A change of coordinates in $\Fq^{k}$ replaces $H$ by $AH$, for some $A\in \GL(k,q)$, and does not change the code $\ker(H)$. A permutation of the vectors $h_i$, or a multiplication of the $h_i$'s by non-zero scalars, corresponds to multiplying $H$ on the right by a monomial matrix. This gives a monomially equivalent code. Conversely, if two codes are monomially equivalent, then their parity-check matrices differ, up to a change of basis of the syndrome space, by a permutation and a non-zero rescaling of the columns. Thus, the construction induces a one-to-one correspondence between the stated equivalence classes.
\end{proof}

\subsection{Equivalent geometric formulations}\label{subsec:geometric_formulation}

We start with an equivalent geometric reformulation of the property of a set of being $(\rho,t)$-saturated, based on subspaces. This allows us to work in the Grassmannian, since we can restrict ourselves to covering $t$-dimensional subspaces with $\rho$-dimensional subspaces generated by the elements of $\mS$. 

Let $\mS\subseteq \Fq^k$. For
$1\le t\le \rho\le k$, define the set of \textbf{good $\rho$-subspaces for $\mS$} as
\[
\mathcal U_\rho(\mS):=
\{U\subseteq \F_q^k \, :\, \dim_{\mathbb F_q}(U)= \rho
\text{ and } U=\langle \mS\cap U\rangle\}.
\]
Thus, $\mathcal U_\rho(\mS)$ is the family of subspaces of $\F_q^k$ of dimension
$\rho$ which are generated by their vectors in $\mS$. Subspaces outside $\mU_\rho(\mS)$ are called \textbf{bad $\rho$-subspaces for $\mS$}.

\begin{lemma}[Subspace-covering criterion]\label{lem:subspace_covering}
Let $\mS\subseteq \Fq^k$, and let
$1\leq t\leq \rho\leq k$.
Then $\mS$ is a $(\rho,t)$-saturating set if and only if for every
$t$-dimensional subspace $T\subseteq \Fq^k$, there exists
$U\in \mathcal U_\rho(\mS)$ such that $T\subseteq U$.
\end{lemma}

\begin{proof}
Assume first that every $t$-dimensional subspace $T\subseteq \Fq^k$ is contained
in some $U\in\mathcal U_\rho(\mS)$. Every such~$U$ is generated by $\rho$ vectors of $\mS$ and every set of $t$ vectors of $\F_q^k$ is contained in a $t$-dimensional subspace of $\Fq^k$. Hence, $\mS$ is $(\rho,t)$-saturating.

Conversely, assume that $\mS$ is $(\rho,t)$-saturating. Then, by definition, every
set of $t$ vectors of $\F_q^k$ is contained in a subspace
$U_0\subseteq \Fq^k$ generated by $\rho$ vectors of $\mS$. In particular, their span must be contained in such subspaces. 
If we only consider sets of $t$ vectors in general position, this implies that for every $t$-dimensional subspace $T$ of $\Fq^k$ there exists $U_0\subseteq \Fq^k$ such that  \[
\dim U_0\leq \rho
\qquad\text{and}\qquad
U_0=\left\langle \mS\cap U_0\right\rangle .
\]
If $\dim U_0=\rho$, then $U_0\in\mathcal U_\rho(\mS)$, and we are done.
If $\dim U_0<\rho$, then  we may choose
vectors of $\mS$ outside $U_0$ and enlarge $U_0$ step by step to a subspace $U\subseteq \Fq^k$ such that 
\[
U_0\subseteq U,
\qquad
\dim U=\rho,
\qquad
U=\left\langle \mS\cap U\right\rangle .
\]
This procedure can be done since $\mS$ must span the entire ambient space $\Fq^k$ due to its saturating property.
Thus $U\in\mathcal U_\rho(\mS)$ and $T\subseteq U$. This proves the claim.
\end{proof}

Using Lemma \ref{lem:subspace_covering}, we immediately observe that $(\rho,t)$-saturating sets can be considered as a generalization of $t$-strong blocking set. This is shown in the following result.

\begin{proposition}\label{prop:strong_(t,t)saturating}
A set $\mS\subseteq \F_q^{k}$ is a $(t,t)$-saturating set if and only if it is a $t$-strong blocking set.
\end{proposition}

\begin{proof}
Assume first that $\mS$ is a $(t,t)$-saturating set. Let $T\subseteq \F_q^k$ be a $t$-dimensional subspace. By the subspace-covering criterion of Lemma~\ref{lem:subspace_covering} there exists a subspace $U\subseteq \F_q^k$ with $\dim U=t$ such that $T\subseteq U$ and $U=\langle \mS\cap U\rangle$.
 Hence, $T=U$ and $T=\langle \mS\cap T\rangle$.
Therefore $\mS$ is a $t$-strong blocking set.

Conversely, assume that $\mS$ is a $t$-strong blocking set and let us show that $\mS$ satisfies the subspace-covering criterion. Let $T$ be a $t$-dimensional subspace of $\F_q^k$. By the $t$-strong blocking property, $T=\langle \mS\cap T\rangle$. Thus, $T\in \mathcal U_{t}(\mS)$.
\end{proof}

As a consequence, we obtain a duality between $(t,t)$-covering codes and $(k-t)$-minimal codes.

\begin{corollary}
    An $[n,n-k]_q$ code $\mC$ is $(t,t)$-covering  if and only if $\mC^\perp$ is $(k-t)$-minimal. 
\end{corollary}
\begin{proof}
    It follows immediately by combining Proposition~\ref{prop:strong_(t,t)saturating} and \cite[Theorem~4.5]{chen2025t}.
\end{proof}


We now give an affine interpretation of $(\rho,t)$-saturating sets, in the spirit of
\cite[Lemma~1.2]{bishnoi2024blocking}.

\begin{proposition}[Affine criterion]\label{prop:affine_rho-t}
Let $\mS\subseteq \Fq^k$, and let $1\le t\le \rho\le k$. Then $\mS$ is
$(\rho,t)$-saturating if and only if for every affine subspace 
$A\subseteq \Fq^k$ of dimension $t-1$ there exists $U\in\mathcal U_\rho(\mS)$ such that $A\subseteq U$.
Equivalently, the family $\mathcal U_\rho(\mS)$ covers all affine
subspaces of $\Fq^k$ of dimension $t-1$ by containment.
\end{proposition}

\begin{proof}
Assume first that $\mS$ is $(\rho,t)$-saturating, and let
$A\subseteq \Fq^k$ be an affine subspace of dimension $t-1$. If $0\notin A$, then
$\langle A\rangle$ is a vector subspace of dimension $t$. Hence, by the $(\rho,t)$-saturating
property, there exists $U\in\mathcal U_\rho(\mS)$ such that $\langle A\rangle\subseteq U$.
In particular, $A\subseteq U$. If $0\in A$, then $A$ is a $(t-1)$-dimensional vector subspace. Since $t\le \rho$, we may extend $A$ to a $t$-dimensional vector subspace $T\subseteq \Fq^k$. Again, by the $(\rho,t)$-saturating property, there exists $U\in\mathcal U_\rho(\mS)$ such that $T\subseteq U$,
and so $A\subseteq U$.

Conversely, assume that every affine subspace of dimension $t-1$ of $\Fq^k$ is contained in some element of~$\mathcal U_\rho(\mS)$. Let $T\subseteq \Fq^k$ be a $t$-dimensional vector subspace. Choose an affine hyperplane $A$ of $T$ not containing the origin.
Then $A$ is an affine subspace of $\Fq^k$ of dimension $t-1$, and $\langle A\rangle=T$.
By assumption, there exists $U\in\mathcal U_\rho(\mS)$ such that $A\subseteq U$.
Since $U$ is a vector subspace, it follows that $T=\langle A\rangle\subseteq U$.
Thus, every~$t$-dimensional vector subspace of $\Fq^k$ is contained in some vector subspace of
dimension at most $\rho$ generated by vectors of $\mS$. This is equivalent to
$\mS$ being $(\rho,t)$-saturating.
\end{proof}

In the special case $\rho=t$, the Proposition~\ref{prop:affine_rho-t} recovers an ordinary
affine blocking condition, which was first found in \cite[Lemma~1.2]{bishnoi2024blocking}. Let $\mS\subseteq \Fq^k$, and let
\[
B_\mS:=\bigcup_{h\in \mS}\langle h\rangle\subseteq \Fq^k
\]
be the cone over $\mS$.

\begin{corollary}\label{cor:affine-blocking-rho=t}
Let $\mS\subseteq \Fq^k$, and let $t\in[k]$. Then $\mS$ is
$(t,t)$-saturating if and only if $B_\mS$ is an affine
$(k-t+1)$-blocking set in $\Fq^k$, i.e. $B_\mS$ meets every affine
subspace of $\Fq^k$ of dimension $t-1$.
\end{corollary}

\begin{proof}
By  Proposition~\ref{prop:affine_rho-t}, $\mS$ is $(t,t)$-saturating if and only if every
affine subspace $A\subseteq \Fq^k$ of dimension $t-1$ is contained in a $t$-dimensional vector
subspace $U$ such that $U=\langle \mS\cap U\rangle$.

Assume first that $\mS$ is $(t,t)$-saturating. Let $A\subseteq \Fq^k$ be an
affine subspace of dimension $t-1$. If $0\in A$, then $A\cap B_\mS\neq\emptyset$, since
$0\in B_\mS$. Suppose then that $0\notin A$, and set $T=\langle A\rangle$.
Then $\dim T=t$. Since $\mS$ is $(t,t)$-saturating, we have $T=\langle \mS\cap T\rangle$.
Write
\[
A=\{x\in T:\ f(x)=c\}
\]
for some nonzero linear functional $f:T\to\mathbb F_q$ and some
$c\in\mathbb F_q^\ast$. Since $\mS\cap T$ spans $T$, there exists
$h\in \mS\cap T$ such that $f(h)\neq0$. Hence
\[
\frac{c}{f(h)}h\in A\cap \langle h\rangle\subseteq A\cap B_\mS.
\]
Thus $B_\mS$ meets every affine subspace of dimension $t-1$.

Conversely, suppose that $B_\mS$ meets every affine subspace of $\Fq^k$ of dimension $t-1$.
Let $T\subseteq \Fq^k$ be a $t$-dimensional subspace. We show that
$\langle \mS\cap T\rangle=T$. If this were not the case, then $\mS\cap T$
would be contained in a hyperplane $\ker(f)$ of $T$, for some nonzero
linear functional $f:T\to\mathbb F_q$. Then the affine hyperplane
\[
A=\{x\in T \,:\, f(x)=1\}
\]
would satisfy $A\cap B_\mS=\emptyset$, a contradiction. Therefore
$\langle \mS\cap T\rangle=T$, and $\mS$ is $(t,t)$-saturating.
\end{proof}

\begin{remark}
In \cite[Lemma~1.2]{bishnoi2024blocking}, the parameter $s$ denotes the affine codimension:
an affine $s$-blocking set in $\mathbb F_q^k$ meets every affine subspace
of codimension $s$. In the present paper, the parameter $t$ denotes the
dimension of the vector subspaces appearing in the definition of $t$-strong
blocking sets. Hence the translation between the two conventions is $s=k-t+1$.
Thus, under our convention, a $(t,t)$-saturating set gives rise to an affine
$(k-t+1)$-blocking set, not to an affine $(t+1)$-blocking set.
\end{remark}

The affine criterion of Proposition~\ref{prop:affine_rho-t} gives a useful
primal reformulation of the $(\rho,t)$-saturating property: one has to cover all affine subspaces of dimension $t-1$ by $\rho$-dimensional subspaces generated by the set. We now propose a dual version of the same condition.
For ease of notation, in this result we set $V=\Fq^k$. Passing to the dual space $V^\vee$, the containment $T\subseteq U$, where $\dim T=t$ and $\dim U=\rho$, becomes $U^\perp \subseteq T^\perp$, with $\dim U^\perp=k-\rho$ and $\dim T^\perp=k-t$.
Thus, instead of asking whether every $t$-dimensional subspace of $V$ is contained in a good $\rho$-dimensional subspace, we may ask whether every $(k-t)$-dimensional subspace of $V^\vee$ contains at least one good $(k-\rho)$-dimensional subspace. The bad subspaces are precisely those whose orthogonal complements are not generated by the vectors of $\mS$ lying inside them. This gives the following dual Grassmannian formulation.

\begin{proposition}[Dual Grassmannian criterion]\label{prop:dual_Grassmannian}
Let $\mS\subseteq V=\Fq^k$, and let $1\leq t\leq \rho\leq k$.
Define
\[
\mathcal B_\rho(\mS):=
\left\{
M\subseteq V^\vee:
\dim M=k-\rho
\text{ and }
M^\perp\neq
\left\langle \mS\cap M^\perp\right\rangle
\right\}.
\]
Then $\mS$ is a $(\rho,t)$-saturating set if and only if for every
subspace $L\subseteq V^\vee$ of dimension $k-t$, there exists a subspace $M\subseteq L$
of dimension $k-\rho$ such that $M\notin \mathcal B_\rho(\mS)$.
Equivalently, $\mS$ is a $(\rho,t)$-saturating set if and only if every  $(k-t)$-dimensional subspace of the dual space
$V^\vee$ contains at least one good 
$(k-\rho)$-dimensional subspace.
\end{proposition}

\begin{proof}
Let $T\subseteq V$ be a $t$-dimensional subspace. Then
$T^\perp\subseteq V^\vee$ has dimension $k-t$. A $\rho$-dimensional subspace $U\subseteq V$
contains $T$ if and only if $U^\perp\subseteq T^\perp$.
Moreover, $\dim U^\perp=k-\rho$.
Thus, choosing a $\rho$-dimensional subspace $U$ containing $T$ is
equivalent to choosing a $(k-\rho)$-dimensional subspace $M\subseteq T^\perp$.
The subspace $U$ is generated by $\mS\cap U$ if and only if
$M=U^\perp$ is not in $\mathcal B_\rho(\mS)$. Therefore every
$t$-dimensional subspace $T$ is contained in an $\mS$-generated
$\rho$-dimensional subspace if and only if every subspace
$L=T^\perp\subseteq V^\vee$ of dimension $k-t$ contains at least one
$(k-\rho)$-dimensional subspace outside $\mathcal B_\rho(\mS)$.
\end{proof}


\section{Lower bound on the size of a \texorpdfstring{$(\rho,t)$-}-saturating set}\label{sec:low_bound}

We now derive a lower bound on the size of $(\rho,t)$-saturating sets. In the
special case $\rho=t$, Proposition~\ref{prop:strong_(t,t)saturating}
identifies $(t,t)$-saturating sets with $t$-strong blocking sets. Moreover,
by Corollary~\ref{cor:affine-blocking-rho=t}, this condition can be translated
into an ordinary affine blocking condition for the cone
\[
B_\mS=\bigcup_{h\in \mS}\langle h\rangle .
\]
Thus, when $\rho=t$, one may obtain lower bounds by applying
Jamison--Brouwer--Schrijver-type lower bounds for affine blocking sets \cite{jamison1977covering,brouwer1978blocking}, as observed in \cite{bishnoi2024blocking}, generalizing a lower bound obtained in \cite{alfarano2022three} for $t=k-1$.

However, for general $\rho>t$ the affine interpretation obtained in
Proposition~\ref{prop:affine_rho-t} is no longer an ordinary point-blocking
condition. It says that every affine $(t-1)$-dimensional affine subspace is contained in a
$\rho$-dimensional subspace generated by vectors of $\mS$. Therefore, the
appropriate object is a covering of affine flats, or equivalently of
$t$-dimensional vector subspaces, by $\mS$-generated $\rho$-spaces.
This suggests a different counting argument, inspired by Denaux's lower bound
for saturating sets; see \cite[Proposition~4.2.1]{denaux2022constructing}. 

Before that, we need the following auxiliary result on $q$-binomial coefficients.

\begin{lemma}\label{lem:qbinom}
For all positive integers $a\geq b\geq t$, we have
$$
\frac{\qqbin{a}{t}_q}{\qqbin{b}{t}_q}\geq q^{t(a-b)}.
$$
\end{lemma}

\begin{proof}
Using the product formula for Gaussian binomial coefficients, we get
$$
\frac{\qqbin{a}{t}_q}{\qqbin{b}{t}_q}
=
\frac{\prod_{i=0}^{t-1}(q^{a-i}-1)}
{\prod_{i=0}^{t-1}(q^{b-i}-1)}
=
\prod_{i=0}^{t-1}\frac{q^{a-i}-1}{q^{b-i}-1}.
$$
We show that each factor is at least $q^{a-b}$. Indeed, for every $i\in\{0,\ldots,t-1\}$, set $A=a-i$ and $B=b-i$. Then $A\geq B$ and
$$
\frac{q^A-1}{q^B-1}\geq q^{A-B}.
$$
This is equivalent to $q^A-1\geq q^{A-B}(q^B-1)=q^A-q^{A-B}$, which holds since $q^{A-B}\geq 1$. Therefore,
$$
\frac{q^{a-i}-1}{q^{b-i}-1}\geq q^{a-b}
$$
for every $i\in\{0,\ldots,t-1\}$. Multiplying these $t$ inequalities, we obtain
$$
\frac{\qqbin{a}{t}_q}{\qqbin{b}{t}_q}\geq q^{t(a-b)}.
$$
\end{proof}

\begin{theorem}\label{thm:lower_bound}
Let $\mS$ be a $(\rho,t)$-saturating set in $\Fq^k$. Then
$$
|\mS| > \frac{\rho}{e}q^{\frac{t(k-\rho)}{\rho}}+ \frac{\rho-1}{2},
$$
where $e$ denotes Euler's number.
\end{theorem}

\begin{proof}
Since $\mS$ is a $(\rho,t)$-saturating set, every $t$-dimensional subspace of
$\Fq^k$ is contained in a subspace of dimension at most $\rho$ generated by
vectors of $\mS$.
We count the $t$-dimensional subspaces of $\Fq^k$. There are $\qqbin{k}{t}_q$
such subspaces. On the other hand, any fixed choice of $\rho$ vectors of $\mS$
spans a subspace of dimension at most~$\rho$, and hence contains at most
$\qqbin{\rho}{t}_q$ many $t$-dimensional subspaces. 
Moreover, since any subset of $\mS$ with size \textit{at most} $\rho$ can be extended to a  subset of $\mS$ of size \textit{exactly} $\rho$, every $t$-dimensional subspace is contained in the span of some subset of $\mS$ of size $\rho$.

Therefore,
$$
\binom{|\mS|}{\rho}\qqbin{\rho}{t}_q\geq \qqbin{k}{t}_q.
$$
Expanding the binomial coefficient we get
$$
\frac{\prod\limits_{i=0}^{\rho-1}(|\mS|-i)}{\rho!}
\geq
\frac{\qqbin{k}{t}_q}{\qqbin{\rho}{t}_q},
$$
and hence, by Lemma~\ref{lem:qbinom},
$$
\prod_{i=0}^{\rho-1}(|\mS|-i)
\geq
\rho!\frac{\qqbin{k}{t}_q}{\qqbin{\rho}{t}_q}
\geq
\rho!q^{t(k-\rho)}.
$$
Taking the $\rho$-th root on both sides, we obtain
$$
\sqrt[\rho]{\prod_{i=0}^{\rho-1}(|\mS|-i)}
\geq
\sqrt[\rho]{\rho!}\,q^{\frac{t(k-\rho)}{\rho}}.
$$
Now we use the standard estimate
$
\sqrt[\rho]{\rho!}>\frac{\rho}{e}
$
and get
$$
\sqrt[\rho]{\prod_{i=0}^{\rho-1}(|\mS|-i)}
>
\frac{\rho}{e}q^{\frac{t(k-\rho)}{\rho}}.
$$
Finally, by the AM-GM inequality,
$$
\sqrt[\rho]{\prod_{i=0}^{\rho-1}(|\mS|-i)}
\leq
\frac{1}{\rho}\sum_{i=0}^{\rho-1}(|\mS|-i)
=
|\mS|-\frac{\rho-1}{2}.
$$
Combining the last two inequalities we conclude that
$$
|\mS|-\frac{\rho-1}{2}
>
\frac{\rho}{e}q^{\frac{t(k-\rho)}{\rho}}.
$$
\end{proof}

Note that if $\rho=t$ we immediately have a lower bound on the size of $t$-strong blocking sets. 
\begin{corollary}\label{cor:lowerbound_strong}
    Let $\mS$ be a $t$-strong blocking set in $\F_q^{k}$. Then 
     $$|\mS|> \frac{t}{e}q^{(k-t)} + \frac{t-1}{2}.$$
\end{corollary}

\begin{remark}\label{rem:lowerbound_strong_bishnoi}
    It must be noted that the bound of Corollary~\ref{cor:lowerbound_strong} is not the best known lower bound for the size of a $t$-strong blocking set. Indeed, as already mentioned at the beginning of this section, in \cite[Remark~4.2]{bishnoi2024blocking} it was shown that every $t$-strong blocking set $\mS\subseteq \Fq^k$ has size  $$|\mS|\ge \frac{q^{k-t+1}}{q-1}t.$$ 
    Despite that, the two bounds have the same order of magnitude in $q$, and differ only by the constant factor $e$ asymptotically.
\end{remark}


\section{Constructions}\label{sec:constructions}
In this section we provide some constructions of  $(\rho,t)$-saturating sets. 

\subsection{Constructions from strong blocking sets}\label{subsec:constr_sbs}
One of the most immediate ways of constructing $(\rho,t)$-saturating sets is with the aid of strong blocking sets. 
There is a first straightforward way to do this, by combining Theorem~\ref{thm:sbs-saturating} with Remark~\ref{rem:saturating}(2). The outcome is the following.

\begin{corollary}\label{cor:DGMP}
    Any $r$-strong blocking set in $\Fq^{k}\subseteq \F_{q^r}^k$ is a $(tr,t)$-saturating set in $\F_{q^r}^{k}$.
\end{corollary}

However, Corollary~\ref{cor:DGMP} can be improved, by providing a direct generalization of Theorem~\ref{thm:sbs-saturating}, which takes into account the fact that we want to saturate all subspaces of
dimension at most $t$.

\begin{theorem}\label{thm:rhosbs-rhpsaturating}
    Let $t,r$ be positive integers and assume that $tr \leq k$. Then any $tr$-strong blocking set in $\F_{q}^{k}\subseteq \F_{q^r}^k$ is a $(tr,t)$-saturating set in $\F_{q^r}^{k}$.
\end{theorem}
\begin{proof}
    Let $X\subseteq \F_{q^r}^{k}$ be a set of size $t$. We first show that there exists a $tr$-dimensional subspace $U$ of $\F_{q}^{k}$ such that $X\subseteq \langle U\rangle_{\F_{q^{r}}}$. Let us define the space $Y=\langle X, X^\sigma,\ldots,X^{\sigma^{r-1}}\rangle_{\F_{q^r}}$, where $\sigma$ is the $q$-Frobenius automorphism. First, observe that $\dim_{\F_{q^r}}(Y)=\kappa \le tr$. Moreover, $Y$ is fixed by $\sigma$, and  therefore there exists a basis of $Y$ made by $\kappa$ elements in $\Fq^{k}$. Thus, we can take $U$ as the $\Fq$-space generated by those vectors, which is then an $\Fq$-subspace of $\Fq^{k}$ of $\Fq$-dimension $\kappa\le tr$. At this point, we use the fact that a $tr$-strong blocking set in $\Fq^k$ generates any $tr$-dimensional $\Fq$-subspace $U$ and we conclude. 
\end{proof}

\begin{remark}
    Let us compare the two constructions of $(tr, t)$-saturating sets obtained in Corollary~\ref{cor:DGMP} and in Theorem~\ref{thm:rhosbs-rhpsaturating}. Assume that we aim at constructing a $(\rho,t)$-saturating set in $\F_{q^{r}}^{k}$. Using Corollary~\ref{cor:DGMP} we need a $r$-strong blocking set in $\Fq^{k}$, while for Theorem~\ref{thm:rhosbs-rhpsaturating} we just need a $tr$-strong blocking set in $\Fq^{k}$. It is straightforward to see that any $r$-strong blocking set is also a $tr$-strong blocking set, for instance by using Remark~\ref{rem:saturating}(2) together with Proposition~\ref{prop:strong_(t,t)saturating}. Moreover, an $r$-strong blocking set is typically larger and has size at least $r q^{k-r}$, while there exist $tr$-strong blocking sets of size approximately $tr(k-tr+1)q^{k-tr}$; see e.g.~\cite{bishnoi2024blocking}. 
\end{remark}

\begin{remark}[Asymptotic tightness of the bound in Theorem~\ref{thm:lower_bound}]
In \cite{bishnoi2024blocking}, using a probabilistic argument, it was shown that there always exists a $t$-strong blocking set in $\Fq^k$ of size at most
$$\frac{q^{k-t+1}-1}{q-1}\cdot \frac{(k-t+1)t+2}{\log_q\left(\frac{q^4}{q^3-q+1}\right)}.$$
This argument is a generalization of those used in \cite{heger2021short,alfarano2024outer} for the case $t=k-1$. 
Thus, for fixed $k$ and $t$, the size of the smallest $t$-strong blocking set in $\Fq^k$ is $\mathcal O(q^{k-t})$, which coincides with the asymptotic of the bound in Remark~\ref{rem:lowerbound_strong_bishnoi} and Corollary~\ref{cor:lowerbound_strong}. 

Moreover, using the construction of Theorem~\ref{thm:rhosbs-rhpsaturating}, for suitable values of $\rho, t$ and $r$, we can get the existence of a $(\rho,t)$-saturating set in $\F_{q^r}^k$ whose size is optimal with respect to the bound of  Theorem~\ref{thm:lower_bound}. More precisely, assume $\rho$ and $t$ are non-zero integers with $t$ dividing $\rho$, and let $r=\frac{\rho}{t}$; then, from a $\rho$-strong blocking set in~$\F_q^{k}$ of size $\mathcal O(q^{k-\rho})$, we get a $(\rho,t)$-saturating set in $\F_{q^r}^{k}$ of size $\mathcal O(q^{k-\rho})=\mathcal O((q^{r})^{\frac{k-\rho}{r}})=\mathcal O((q^{r})^{t\frac{k-\rho}{\rho}})$, which indeed meets the lower bound of Theorem~\ref{thm:lower_bound}.

\end{remark}

\subsection{Projective and graph-based constructions}\label{subsec:graph_constr}

In this subsection, we give systematic constructions of $(k-1,t)$-saturating sets in the projective space $\PG(k-1,q)$, generalizing a well-known construction given by the generalized tetrahedron of order $k-t+1$. We use the projective language here because it is more natural in the framework of the proposed construction. In particular, a projective point set is identified with any choice of nonzero representatives in $\F_q^k$; the saturation property is invariant under rescaling of the representatives.

Recall that a \textbf{generalized tetrahedron of order $s$} in $\PG(k-1,q)$ is given as follows. Start with $P_1,\ldots,P_k$ points in general position, and define the set
$$\mathcal T_{s,k}=\bigcup_{\substack{I\subseteq [k] \\ |I|=s}}\langle P_i\,:\, i \in I \rangle.$$

It was shown in \cite{davydov2011linear} that the generalized tetrahedron $\mathcal T_{k-t+1,k}$ is  a $t$-strong blocking set, that is, a $(t,t)$-saturating set. In the following, we will adapt this construction for $(k-1,t)$-saturating sets, by only selecting some of the lines of $\mT_{2,k}$. While for $t=k-1$ it can be shown that we actually need all the $\binom{k}{2}$ lines that constitute $\mT_{2,k}$, for $t<k-1$ we can only select a subset of them. In particular, the procedure for selecting such subsets of lines will rely on the properties of a given graph $G$ on $k$ vertices. A similar idea was already used in \cite{alon2024strong} for constructing $(k-1)$-strong blocking sets, but using graphs on $n>k$ vertices.

Let $\mathrm{PG}(V)^\vee:=\mathrm{PG}(V^\vee)$ denote the
dual projective space, whose points correspond to the hyperplanes of
$\mathrm{PG}(V)$. We first specialize the Dual Grassmannian criterion of Proposition~\ref{prop:dual_Grassmannian} to the case $\rho=k-1$. This is particularly helpful in this case because the dual Grassmannian of hyperplanes is isomorphic to a projective space.

\begin{lemma}[A dual criterion for \texorpdfstring{$(k-1,t)$}{(k-1,t)}-saturation]\label{lem:dual_hyperplane_criterion}
Let $\mS\subseteq \PG(k-1,q)$, and let
\[
B(\mS):=
\left\{
\Sigma\in \PG(k-1,q)^\vee :
\left\langle \mS\cap \Sigma \right\rangle \neq \Sigma
\right\}
\]
be the set of hyperplanes which are not generated by their intersection with $\mS$.
Then, for $t\in[k-1]$, the following are equivalent:
\begin{enumerate}
    \item[(1)] $\mS$ is a $(k-1,t)$-saturating set;
    \item[(2)] $B(\mS)$ contains no projective $(k-t-1)$-subspace of
    $\PG(k-1,q)^\vee$.
\end{enumerate}
\end{lemma}

\begin{proof}
We recall that projectively $\mS$ is $(k-1,t)$-saturating if and only if every
projective $(t-1)$-subspace $\Lambda\subseteq \PG(k-1,q)$
is contained in a projective $(k-2)$-subspace generated by points of
$\mS$. Equivalently, for every such $\Lambda$, there exists a hyperplane
$\Sigma$ of $\PG(k-1,q)$ such that $\Lambda\subseteq \Sigma$ and $\left\langle \mS\cap \Sigma \right\rangle=\Sigma$.

Fix now a projective $(t-1)$-subspace $\Lambda$. The set of all
hyperplanes containing $\Lambda$ is a projective subspace of the dual
space $\PG(k-1,q)^\vee$, namely
\[
\Lambda^\perp :=
\left\{
\Sigma\in \PG(k-1,q)^\vee : \Lambda\subseteq \Sigma
\right\}.
\]
Its dimension is
\[
\dim(\Lambda^\perp) = (k-1)-(t-1)-1=k-t-1.
\]

Therefore, the $(k-1,t)$-saturating property is equivalent to requiring
that, for every projective $(t-1)$-subspace $\Lambda$,
the dual flat $\Lambda^\perp$ contains at least one hyperplane
$\Sigma$ such that $\left\langle \mS\cap \Sigma \right\rangle=\Sigma$.
By definition of $B(\mS)$, this is equivalent to $\Lambda^\perp \not\subseteq B(\mS)$, for every projective $(t-1)$-subspace $\Lambda$.

Finally, as $\Lambda$ varies among the projective $(t-1)$-subspaces of
$\PG(k-1,q)$, the subspaces $\Lambda^\perp$ are precisely the projective
$(k-t-1)$-subspaces of the dual projective space. Hence the previous
condition is equivalent to saying that $B(\mS)$ contains no projective
$(k-t-1)$-subspace of $\PG(k-1,q)^\vee$.
\end{proof}

We now specialize Lemma~\ref{lem:dual_hyperplane_criterion} to a family of
sets arising from graphs.

Let $P_1,\dots,P_k$ be points in general position in $\PG(k-1,q)$, and let
$G$ be a graph on the vertex set~$[k]$. We associate with $G$ the
point set
\[
\mS_G :=
\bigcup_{\{i,j\}\in E(G)}
\langle P_i,P_j\rangle
\subseteq \PG(k-1,q),
\]
where $E(G)$ denotes the edge set of $G$. It can be easily seen that with this notation, when $G$ is the complete graph $K_k$ on $k$ vertices, we recover the generalized tetrahedron $\mT_{2,k}$ of order $2$. Throughout this paragraph we assume
that $G$ has no isolated vertices. This ensures that every point $P_i$
belongs to $\mS_G$.
Choose coordinates so that $P_i=[e_i]$, where $e_1,\dots,e_k$ is the standard basis of $\F_q^k$. For a point $[c]=[c_1:\cdots:c_k]\in \PG(k-1,q)^\vee$, we define
\[
\supp(c):=\{i\in [k]: c_i\neq 0\}.
\]
We also define the set of \textbf{graph-bad hyperplanes} by
\[
B_G :=
\left\{
[c] \in \PG(k-1,q)^\vee :
G[\supp(c)] \text{ is disconnected}
\right\},
\]
where $G[\supp(c)]$ denotes the subgraph of $G$ induced by the vertices in
$\supp(c)$.

\begin{corollary}[Graph form of the dual criterion]\label{cor:graph_dual}

With the notation above, the set $\mS_G$ is a $(k-1,t)$-saturating set if and only if $B_G$ contains no projective $(k-t-1)$-subspace of $\PG(k-1,q)^\vee$.
\end{corollary}

\begin{proof}
By Lemma~\ref{lem:dual_hyperplane_criterion}, it is enough to identify the bad hyperplanes for the set $\mS_G$. 
For any $c\in \Fq^k\setminus\{0\}$  let
\[
\Sigma_c=\{ [x]\in\PG(k-1,q)\,:\,
c_1x_1+\cdots+c_kx_k=0\}
\]
 be a hyperplane of $\PG(k-1,q)$, and set
\[
I:=\supp(c).
\]
We first show that
$\left\langle \mS_G\cap \Sigma_c\right\rangle=\Sigma_c$ if and only if $G[I]$ is connected. Indeed, for every edge $\{i,j\}\in E(G[I])$, the line
$\langle P_i,P_j\rangle$ meets $\Sigma_c$ in the point $[c_j e_i-c_i e_j]$.
If $G[I]$ is connected, these points span the hyperplane
\[
\left\{
\sum_{i\in I} x_i e_i :
\sum_{i\in I} c_i x_i=0
\right\}
\]
inside $\langle e_i : i\in I\rangle$. Moreover, since $G$ has no
isolated vertices, every point $P_h$ with $h\notin I$ belongs to
$\mS_G$, and clearly $P_h\in \Sigma_c$. Hence
\[
\left\langle \mS_G\cap \Sigma_c\right\rangle=\Sigma_{c}.
\]

Conversely, suppose that $G[I]$ is disconnected, and let
$I=I_1\sqcup I_2$ be a nontrivial partition such that there are no edges of
$G[I]$ between $I_1$ and $I_2$. The points of $\mS_G\cap\Sigma_c$ whose
support is contained in $I$ arise only from edges contained in $I_1$ or in
$I_2$. Hence their span inside $\langle P_i:i\in I\rangle$ is contained in
\[
\left\{
\sum_{i\in I}x_ie_i:
\sum_{i\in I_1}c_ix_i=0,\ 
\sum_{i\in I_2}c_ix_i=0
\right\},
\]
which is a proper subspace of
\[
\left\{
\sum_{i\in I}x_ie_i:
\sum_{i\in I}c_ix_i=0
\right\}.
\]
Adding the coordinate points $P_h$, with $h\notin I$, only contributes the
coordinate directions outside $\langle P_i:i\in I\rangle$. Therefore
\[
\langle \mS_G\cap\Sigma_c\rangle\neq \Sigma_c.
\]
Hence, the set of bad hyperplanes $B(\mS_G)$ coincides with $B_G$.
The result now follows immediately from Lemma~\ref{lem:dual_hyperplane_criterion} applied to
$\mS=\mS_G$.
\end{proof}

We now use Corollary~\ref{cor:graph_dual} to give a more explicit construction of $(k-1,t)$-saturating sets in $\PG(k-1,q)$. To this aim we remind here that the \textbf{vertex-connectivity}, $\kappa(G)$, of $G$ is the minimum number of vertices that must be removed from the set $[k]$ to disconnect $G$ or to reduce it to a single isolated vertex.

\begin{theorem}\label{thm:graph_constr}
Let $t\in [k-1]$, and let $G$ be a graph on $[k]$ with $\kappa(G)\ge t$. Let $P_1,\ldots,P_k$ be points in general position in $\PG(k-1,q)$, and let
\[
\mS_G=\bigcup_{\{i,j\}\in E(G)}\langle P_i,P_j\rangle .
\]
Then $\mS_G$ is a $(k-1,t)$-saturating set in $\PG(k-1,q)$.
\end{theorem}

\begin{proof}
By Corollary~\ref{cor:graph_dual}, it is enough to show that $B_G$ contains no projective $(k-t-1)$-subspace of $\PG(k-1,q)^\vee$.
Assume, by contradiction, that $B_G$ contains a projective
$(k-t-1)$-subspace, and let $W\subseteq \mathbb F_q^k$ be the corresponding vector
subspace. Then $\dim W=k-t$
and, for every nonzero $w\in W$, the graph
$G[\supp(w)]$ is disconnected.
Since $\kappa(G)\ge t$, whenever $G[I]$ is disconnected, the complement
$V(G)\setminus I$ has size at least $t$. Hence, $|\supp(w)|\le k-t$ for every nonzero $w\in W$.
Thus $W$ is a linear subspace of dimension $k-t$ whose vectors all have Hamming weight at most $k-t$. By the classification of maximal linear anticodes in the Hamming metric (see e.g. \cite{ravagnani2016generalized}), $W$ is a coordinate subspace. Hence $W$ contains a coordinate vector $e_i$.
But $G[\supp(e_i)]=G[\{i\}]$ is connected, so $[e_i]\notin B_G$, contradicting $\PG(W)\subseteq B_G$.
Therefore $B_G$ contains no projective $(k-t-1)$-subspace. By
Corollary~\ref{cor:graph_dual}, $\mS_G$ is a $(k-1,t)$-saturating set.
\end{proof}

A concrete family of graphs satisfying the hypothesis of Theorem~\ref{thm:graph_constr} is obtained by deleting complete graphs on pairwise disjoint vertex sets from $K_k$.

\begin{corollary}\label{cor:general_graph}
Let $P_1,\ldots,P_k$ be $k$ points in general position in $\PG(k-1,q)$. Let $A_1,\ldots,A_s\subseteq [k]$ be
pairwise disjoint subsets with $|A_a|\le k-t$ for every $a \in [s]$. Let $G$ be the graph obtained from $K_k$ by deleting, for every $a$, all edges of the complete graph on $A_a$, and define
\[
\mS_G=
\bigcup_{\{i,j\}\in E(G)}
\langle P_i,P_j\rangle .
\]
Then $\mS_G$ is a $(k-1,t)$-saturating set.
\end{corollary}

\begin{proof}
Let
\[
r_{\max}:=\max_{a\in [s]} |A_a|.
\]
The graph $G$ is obtained from $K_k$ by deleting only the edges inside the
sets $A_a$. Hence all edges between distinct sets $A_a,A_b$, and all edges
from $A_a$ to its complement, are still present.
It follows that a vertex cut disconnecting $G$ must contain all vertices outside some $A_a$. Indeed, if after removing a set $X$ there remain vertices in two different sets $A_a,A_b$, or one vertex in some $A_a$ and one outside $A_a$, then these vertices are adjacent in $G$. Therefore
\[
\kappa(G)=k-r_{\max}.
\]
Since $r_{\max}\le k-t$, we have
\[
\kappa(G)\ge t.
\]
The conclusion follows from Theorem~\ref{thm:graph_constr}.
\end{proof}

We illustrate the construction in $\PG(3,q)$ and $\PG(5,q)$ in the following examples. The case of $\PG(4,q)$ is analogous to the first example: one starts from five points in general position and removes the lines determined by a pairwise disjoint family of subsets of size at most $2$.

\begin{example}[The tetrahedron in $\PG(3,q)$]
Let $P_1,P_2,P_3,P_4$ be four points in general position in $\PG(3,q)$.
They determine the tetrahedron whose edges are the six projective lines
$\langle P_i,P_j\rangle$, with $1\leq i<j\leq 4$.

Consider the two disjoint subsets
\[
A_1=\{1,2\},
\qquad
A_2=\{3,4\}.
\]
We remove the two lines determined by these subsets, namely
$\langle P_1,P_2\rangle$ and $\langle P_3,P_4\rangle$, and keep the four
remaining lines. Thus we define
\[
\mS
=
\langle P_1,P_3\rangle
\cup
\langle P_1,P_4\rangle
\cup
\langle P_2,P_3\rangle
\cup
\langle P_2,P_4\rangle .
\]
Then $\mS$ is a $(3,2)$-saturating set in $\PG(3,q)$.
Indeed, this is the case $k=4$ and $t=2$ of
Corollary~\ref{cor:general_graph}: here the deleted complete graphs are
supported on pairwise disjoint subsets of size at most $k-t=2$.
More explicitly, every line of $\PG(3,q)$ is contained in a plane generated
by the points of $\mS$ lying in that plane. Hence, even after removing two
opposite lines from the tetrahedron, the remaining four lines are still
enough to obtain a $(3,2)$-saturating set; see Figure~\ref{fig:tetrahedron-opposite-lines}.

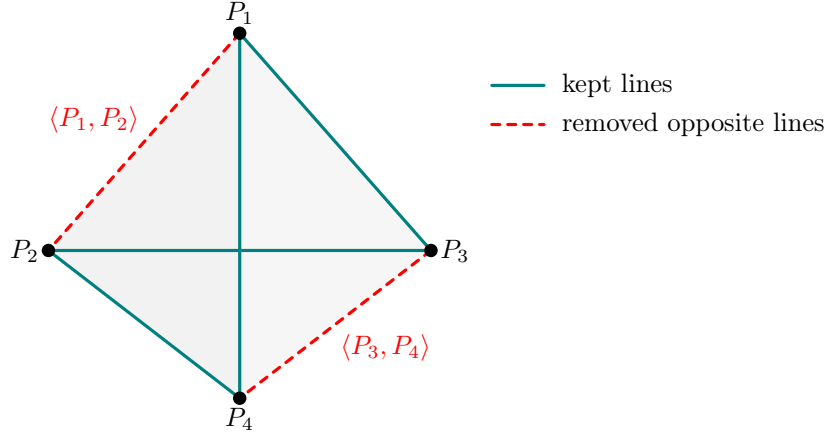
\begin{figure}[ht]
\centering
\begin{tikzpicture}[scale=1.15, line cap=round, line join=round]

\coordinate (P1) at (0,2.5);
\coordinate (P2) at (-2.2,0);
\coordinate (P3) at (2.2,0);
\coordinate (P4) at (0,-1.7);

\fill[gray!10] (P1) -- (P2) -- (P4) -- cycle;
\fill[gray!7]  (P1) -- (P3) -- (P4) -- cycle;

\draw[red, very thick, dashed] (P1) -- (P2);
\draw[red, very thick, dashed] (P3) -- (P4);

\draw[teal, very thick] (P1) -- (P3);
\draw[teal, very thick] (P1) -- (P4);
\draw[teal, very thick] (P2) -- (P3);
\draw[teal, very thick] (P2) -- (P4);

\fill (P1) circle (2.2pt);
\fill (P2) circle (2.2pt);
\fill (P3) circle (2.2pt);
\fill (P4) circle (2.2pt);

\node[above] at (P1) {$P_1$};
\node[left] at (P2) {$P_2$};
\node[right] at (P3) {$P_3$};
\node[below] at (P4) {$P_4$};

\node[red, above left] at (-1.05,1.25) {$\langle P_1,P_2\rangle$};
\node[red, below right] at (1.05,-0.85) {$\langle P_3,P_4\rangle$};

\begin{scope}[shift={(2.9,1.9)}]
  \draw[teal, very thick] (0,0) -- (0.6,0);
  \node[right] at (0.7,0) {kept lines};

  \draw[red, very thick, dashed] (0,-0.45) -- (0.6,-0.45);
  \node[right] at (0.7,-0.45) {removed opposite lines};
\end{scope}

\end{tikzpicture}
\caption{A tetrahedron in $\PG(3,q)$. The two dashed red lines are removed, while the four colored lines form a $(3,2)$-saturating set.}
\label{fig:tetrahedron-opposite-lines}
\end{figure}
\end{example}

\begin{example}[The tetrahedron in $\PG(5,q)$]
Let $P_1,\ldots,P_6$ be six points in general position in $\PG(5,q)$.
They determine the projective simplex whose edges are the lines
$\langle P_i,P_j\rangle$, with $1\leq i<j\leq 6$.
We illustrate two different applications of Corollary~\ref{cor:general_graph}.

First, consider the three disjoint subsets
\[
A_1=\{1,2\},
\qquad
A_2=\{3,4\},
\qquad
A_3=\{5,6\}.
\]
Remove the three pairwise disjoint lines
\[
\langle P_1,P_2\rangle,\qquad
\langle P_3,P_4\rangle,\qquad
\langle P_5,P_6\rangle,
\]
and keep all the remaining lines. Thus
\[
\mS_1
=
\bigcup_{\substack{1\leq i<j\leq 6\\
\{i,j\}\not\subseteq A_a\ \text{for every }a}}
\langle P_i,P_j\rangle .
\]
Since this corresponds to removing complete graphs supported on disjoint
subsets of size $2$, the case $k=6$ and $t=4$ of
Corollary~\ref{cor:general_graph} shows that $\mS_1$ is a $(5,4)$-saturating set in $\PG(5,q)$.

Moreover, the same set $\mS_1$ is also $(5,3)$-saturating. Indeed, for
$k=6$ and $t=3$, the deleted complete graphs are still supported on subsets
of size at most $k-t=3$. Therefore Corollary~\ref{cor:general_graph} also
applies. Equivalently, this follows from monotonicity: every $(5,4)$-saturating set is automatically $(5,3)$-saturating.

Second, consider the two disjoint triples
\[
B_1=\{1,2,3\},
\qquad
B_2=\{4,5,6\}.
\]
Remove all the lines joining pairs of points inside $B_1$ and inside $B_2$,
namely
\[
\langle P_1,P_2\rangle,\quad
\langle P_1,P_3\rangle,\quad
\langle P_2,P_3\rangle
\]
and
\[
\langle P_4,P_5\rangle,\quad
\langle P_4,P_6\rangle,\quad
\langle P_5,P_6\rangle .
\]
Keep all the remaining lines, and define
\[
\mS_2
=
\bigcup_{\substack{1\leq i<j\leq 6\\
\{i,j\}\not\subseteq B_1,\ \{i,j\}\not\subseteq B_2}}
\langle P_i,P_j\rangle .
\]
Equivalently,
\[
\mS_2
=
\bigcup_{\substack{i\in B_1\\ j\in B_2}}
\langle P_i,P_j\rangle .
\]
Thus, $\mS_2$ is the union of the nine lines joining a point of
$\{P_1,P_2,P_3\}$ to a point of $\{P_4,P_5,P_6\}$.
Since this corresponds to removing complete graphs supported on two disjoint
subsets of size $3$, the case $k=6$ and $t=3$ of
Corollary~\ref{cor:general_graph} shows that $\mS_2$ is a
$(5,3)$-saturating set in $\PG(5,q)$.
Geometrically, the $(5,4)$-saturating property means that every projective
$3$-space of $\PG(5,q)$ is contained in a hyperplane generated by the points
of the set lying in that hyperplane. The $(5,3)$-saturating property means
that every projective plane of $\PG(5,q)$ is contained in such a hyperplane.
Thus, in $\PG(5,q)$, removing three disjoint lines gives a configuration
which is both $(5,4)$- and $(5,3)$-saturating, while removing two disjoint
triangles gives a more economical configuration tailored to the
$(5,3)$-saturating property; see Figure~\ref{fig:pg5-clique-removal}.

\begin{figure}[ht]
\centering

\begin{minipage}{0.49\textwidth}
\centering
\begin{tikzpicture}[scale=0.72, line cap=round, line join=round]

\coordinate (P1) at (0,2.8);
\coordinate (P2) at (-2.5,1.0);
\coordinate (P3) at (2.5,1.0);
\coordinate (P4) at (-2.0,-1.5);
\coordinate (P5) at (2.0,-1.5);
\coordinate (P6) at (0,-2.8);

\fill[gray!8] (P1) -- (P2) -- (P4) -- cycle;
\fill[gray!6] (P1) -- (P3) -- (P5) -- cycle;
\fill[gray!5] (P2) -- (P3) -- (P6) -- cycle;

\draw[teal, very thick] (P1) -- (P3);
\draw[teal, very thick] (P1) -- (P4);
\draw[teal, very thick] (P1) -- (P5);
\draw[teal, very thick] (P1) -- (P6);

\draw[teal, very thick] (P2) -- (P3);
\draw[teal, very thick] (P2) -- (P4);
\draw[teal, very thick] (P2) -- (P5);
\draw[teal, very thick] (P2) -- (P6);

\draw[teal, very thick] (P3) -- (P5);
\draw[teal, very thick] (P3) -- (P6);

\draw[teal, very thick] (P4) -- (P5);
\draw[teal, very thick] (P4) -- (P6);

\draw[red, very thick, dashed] (P1) -- (P2);
\draw[red, very thick, dashed] (P3) -- (P4);
\draw[red, very thick, dashed] (P5) -- (P6);

\fill (P1) circle (2.2pt);
\fill (P2) circle (2.2pt);
\fill (P3) circle (2.2pt);
\fill (P4) circle (2.2pt);
\fill (P5) circle (2.2pt);
\fill (P6) circle (2.2pt);

\node[above] at (P1) {$P_1$};
\node[left] at (P2) {$P_2$};
\node[right] at (P3) {$P_3$};
\node[left] at (P4) {$P_4$};
\node[right] at (P5) {$P_5$};
\node[below] at (P6) {$P_6$};

\node at (0,-3.65) {\small removing three disjoint lines};

\end{tikzpicture}
\end{minipage}
\hfill
\begin{minipage}{0.49\textwidth}
\centering
\begin{tikzpicture}[scale=0.72, line cap=round, line join=round]

\coordinate (P1) at (0,2.8);
\coordinate (P2) at (-2.5,1.0);
\coordinate (P3) at (2.5,1.0);
\coordinate (P4) at (-2.0,-1.5);
\coordinate (P5) at (2.0,-1.5);
\coordinate (P6) at (0,-2.8);

\fill[gray!8] (P1) -- (P2) -- (P4) -- cycle;
\fill[gray!6] (P1) -- (P3) -- (P5) -- cycle;
\fill[gray!5] (P2) -- (P3) -- (P6) -- cycle;

\draw[teal, very thick] (P1) -- (P4);
\draw[teal, very thick] (P1) -- (P5);
\draw[teal, very thick] (P1) -- (P6);

\draw[teal, very thick] (P2) -- (P4);
\draw[teal, very thick] (P2) -- (P5);
\draw[teal, very thick] (P2) -- (P6);

\draw[teal, very thick] (P3) -- (P4);
\draw[teal, very thick] (P3) -- (P5);
\draw[teal, very thick] (P3) -- (P6);

\draw[red, very thick, dashed] (P1) -- (P2);
\draw[red, very thick, dashed] (P1) -- (P3);
\draw[red, very thick, dashed] (P2) -- (P3);

\draw[red, very thick, dashed] (P4) -- (P5);
\draw[red, very thick, dashed] (P4) -- (P6);
\draw[red, very thick, dashed] (P5) -- (P6);

\fill (P1) circle (2.2pt);
\fill (P2) circle (2.2pt);
\fill (P3) circle (2.2pt);
\fill (P4) circle (2.2pt);
\fill (P5) circle (2.2pt);
\fill (P6) circle (2.2pt);

\node[above] at (P1) {$P_1$};
\node[left] at (P2) {$P_2$};
\node[right] at (P3) {$P_3$};
\node[left] at (P4) {$P_4$};
\node[right] at (P5) {$P_5$};
\node[below] at (P6) {$P_6$};

\node at (0,-3.65) {\small removing two disjoint triangles};

\end{tikzpicture}
\end{minipage}

\caption{Two clique-removal constructions in $\PG(5,q)$ using the same six points in general position. On the left, removing three disjoint lines gives a configuration which is $(5,4)$-saturating. On the right, removing two disjoint triangles gives a $(5,3)$-saturating set.}
\label{fig:pg5-clique-removal}
\end{figure}
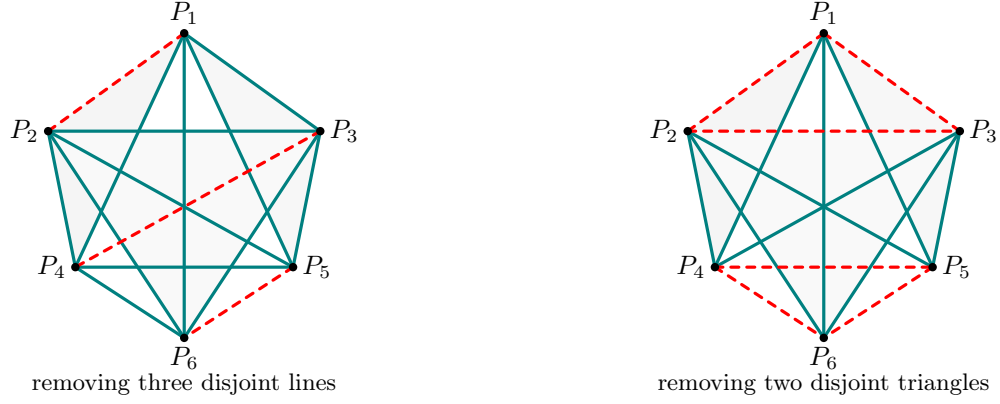
\end{example}

\begin{remark}
The size of the saturating set constructed with Corollary~\ref{cor:general_graph} can be computed explicitly. Since the points
$P_1,\ldots,P_k$ are in general position, two distinct lines
$\langle P_i,P_j\rangle$ and $\langle P_a,P_b\rangle$ meet only when the
corresponding edges share a vertex. Hence, if $G$ has no isolated vertices,
then
\[
|\mS_G|=k+(q-1)|E(G)|.
\]
In particular,if $\kappa(G)\geq t$, then every vertex of $G$ has degree at least $t$, and hence
\[
|\mS_G|\geq k+\frac{kt}{2}(q-1),
\]
with equality when $G$ is $t$-regular.
If $G$ is obtained from $K_k$ by deleting complete graphs on pairwise
disjoint vertex sets $A_1,\ldots,A_s$, then
\[
|\mS_G|
=
k+(q-1)\left(
\binom{k}{2}-\sum_{a=1}^s\binom{|A_a|}{2}
\right).
\]
For fixed $k$ and $t$, this size is linear in $q$. On the other hand, the
general lower bound of Theorem~\ref{thm:lower_bound}, applied with $\rho=k-1$ and $t$, gives
a lower bound of order
\[
q^{\frac{t}{k-1}}.
\]
Thus, the graph-based construction is not asymptotically optimal in general,
but it is explicit and geometrically simple. Its validity relies on the dual
bad-hyperplane criterion of Corollary~\ref{cor:graph_dual}.
\end{remark}

\medskip

\bibliographystyle{abbrv}
\bibliography{biblio}

\end{document}